\documentclass[reqno]{amsart}
\usepackage{amsfonts, amsmath, amsthm, amssymb}
\usepackage{mathrsfs, graphicx}

\usepackage[all]{xy}

\newtheorem{thm}{Theorem}[section]
\newtheorem{lem}[thm]{Lemma}
\newtheorem{lemma}[thm]{Lemma}

\newtheorem{cor}[thm]{Corollary}
\newtheorem{ex}[thm]{Example}
\newtheorem{dff}[thm]{Definition}
\newtheorem{remark}[thm]{Remark}
\newtheorem{question}[thm]{Question}
\newtheorem{conjecture}[thm]{Conjecture}

\numberwithin{thm}{subsection}

\usepackage{color}

\newcommand{\GL}{{\operatorname{GL}}}
\newcommand{\Ccut}{{\Sigma}_{\mathrm{cut}}}
\newcommand{\Aut}{\operatorname{Aut}}
\newcommand{\Hom}{\operatorname{Hom}}
\newcommand{\End}{\operatorname{End}}
\newcommand{\Tr}{\operatorname{Tr}}

\title[Representations with finite mapping class group orbit]{Representations of surface groups with universally finite mapping class group orbit}
\author{Brian Lawrence and Daniel Litt}

\begin{document}

\maketitle

\begin{abstract}
   Let $\Sigma_{g,n}$ be the orientable genus $g$ surface with $n$ punctures, where $2-2g-n<0$. Let $$\rho: \pi_1(\Sigma_{g,n})\to GL_m(\mathbb{C})$$ be a representation. Suppose that for each finite covering map $f: \Sigma_{g', n'}\to \Sigma_{g, n}$, the orbit of (the isomorphism class of) $f^*(\rho)$ under the mapping class group $MCG(\Sigma_{g',n'})$ of $\Sigma_{g',n'}$ is finite. Then we show that $\rho$ has finite image. The result is motivated by the Grothendieck-Katz $p$-curvature conjecture, and gives a reformulation of the $p$-curvature conjecture in terms of isomonodromy.
\end{abstract}

\section{Introduction}
\subsection{The main result}
%
%
%
%
%
%
The goal of this paper is to prove a result on mapping class group actions on character varieties, motivated by questions from algebraic and arithmetic geometry. 

Our main result may be stated purely topologically. Let $\Sigma$ be an orientable surface (possibly with finitely many punctures and boundary components) with $\chi(\Sigma)<0$. Note that the mapping class group $\text{MCG}(\Sigma)$ of $\Sigma$ has a natural outer action on $\pi_1(\Sigma)$, and hence acts on the set of isomorphism classes of complex representations of $\pi_1(\Sigma)$.
\begin{thm}\label{main-thm}
Let $$\rho: \pi_1(\Sigma)\to GL_m(\mathbb{C})$$ be a representation. Suppose that for each finite covering map $$f: \Sigma'\to \Sigma,$$ the orbit of (the isomorphism class of) $f^*(\rho)$ under the mapping class group $MCG(\Sigma')$ is finite. Then $\rho$ has finite image.
\end{thm}
We will give a proof of Theorem \ref{main-thm} in Section \ref{thm-pf}.
\begin{remark}
Note that if $\rho: \pi_1(\Sigma)\to GL_m(\mathbb{C})$ has finite image, then its orbit under the mapping class group is finite. In general, the converse is not true; see e.g. \cite[Proposition 1.2]{BKMS} or Example \ref{parshin-example} of this paper.  See also \cite[Theorem 1.1]{BKMS} for a result related to our Theorem \ref{main-thm}, where the mapping class group is replaced by $\text{Aut}(\pi_1(\Sigma))$.

See also \cite{BGMW} for stronger results in the case of representations into $SL_2(\mathbb{C})$.
\end{remark}
As a corollary of our main theorem, we have the following purely group-theoretic statement:
\begin{cor}\label{main-cor}
Let $$\rho:\pi_1(\Sigma)\to GL_m(\mathbb{C})$$ be a representation. Suppose that for each finite index subgroup $G\subset \pi_1(\Sigma)$, the orbit of (the isomorphism class) of $\rho|_G$ under $\text{Out}(G)$ is finite. Then $\rho$ has finite image.
\end{cor}
\begin{remark}
Note that the analogue of Corollary \ref{main-cor} is not true for general groups. For example, let $n>2$ and let $$\rho_{\text{std}}: SL_n(\mathbb{Z}) \to GL_n(\mathbb{C})$$ be the standard representation. For any $G\subset SL_n(\mathbb{Z})$ of finite index, the orbit of $\rho_{\text{std}}|_G$ under $\text{Out}(G)$ is a singleton, by e.g.~Margulis super-rigidity --- but of course $\rho_{\text{std}}$ has infinite image.
\end{remark}

\begin{remark}
To clarify ideas, we'll explain the special case where the representation $\rho$ has rank $m=1$, i.e.\ is given by a map
\[ \rho \colon \pi_1(\Sigma)\to \mathbb{C}^*. \]
Such a $\rho$ must factor through the abelianization $H_1(\Sigma)$ of $\pi_1(\Sigma)$.
Choosing a basis for $H_1(\Sigma)$, we see that the set of such $\rho$ is in bijection with
\[ \operatorname{Hom}(H_1(\Sigma), (\mathbb{C}^*)) \cong (\mathbb{C}^*)^{2g} \cong (\mathbb{C}/\mathbb{Z})^{2g} \]
(the second isomorphism being given by a suitably normalized logarithm).

The mapping class group acts through its quotient $\operatorname{Sp}_{2g}(\mathbb{Z})$ on $(\mathbb{C}/\mathbb{Z})^{2g}$ in the obvious way.
In order that $\rho$ be MCG-finite, 
the corresponding point of $(\mathbb{C}/\mathbb{Z})^{2g}$
must have finite orbit under the action of $\operatorname{Sp}_{2g}(\mathbb{Z})$.
One verifies that the only such points are torsion points, i.e.\ elements of
\[ (\mathbb{Q}/\mathbb{Z})^{2g}. \]
Hence, if $\rho$ is MCG-finite, then it has finite image.
\end{remark}

Theorem \ref{main-thm} and Corollary \ref{main-cor} are in fact equivalent in the case that $\Sigma$ is a closed surface, as in this case $MCG(\Sigma)$ has finite index in $\text{Out}(\pi_1(\Sigma))$ by the Dehn-Nielsen-Baer Theorem. The case of surfaces with punctures (and possibly boundary components) also admits a purely group-theoretic reformulation, but we omit it here.

For the rest of the introduction, we explain the motivation for this theorem, arising from the $p$-curvature conjecture, and its implications for isomonodromic deformations of flat vector bundles on algebraic curves.
\subsection{The algebro-geometric setting} Let $C$ be smooth proper algebraic curve over the field of complex numbers, and let $D\subset C$ be a finite set.  The Riemann-Hilbert correspondence is an equivalence of categories between the category of algebraic flat vector bundles with regular singularities at infinity on $C\setminus D$, (that is, flat vector bundles on $C\setminus D$ which extend to objects of the category $$\text{MIC}(C(\log D))$$ of 
vector bundles with flat holomorphic connection $$(\mathscr{E}, \nabla: \mathscr{E}\to \mathscr{E}\otimes \Omega^1_C(\log D))$$ on $C$) and the category $\text{LocSys}(C\setminus D)$ of
complex local systems on $C\setminus D$.
If we choose a base-point $x \in C$, then monodromy gives an equivalence of both categories above with the category $\text{Rep}_{\mathbb{C}}(\pi_1(C\setminus D, x))$ of
representations $$\rho: \pi_1(C\setminus D, x) \rightarrow \operatorname{GL}_n(\mathbb{C})$$
of the topological fundamental group of $C$. Let $\rho_{\mathscr{E}, \nabla}$ be the representation associated to a flat vector bundle $(\mathscr{E}, \nabla)$.

Consider the relative situation, where we have a family $\pi: \mathcal{C} \rightarrow S$
of smooth proper curves over a smooth base $S$, which we take to be a scheme over $\mathbb{C}$.
Locally for the complex topology, we can choose a section $x: S \rightarrow \mathcal{C}$,
and (the isomorphism class of) the fundamental group of the fiber is locally constant on $S$.
If we are given a base-point $s_0 \in S$, and a vector bundle with connection
$$(\mathscr{E}_0, \nabla_0: \mathscr{E}_0\to \mathscr{E}_0\otimes \Omega^1_C)$$ on the fiber $C_0 = \mathcal{C}_{s_0}$,
there is a unique (up to canonical isomorphism) analytic deformation $$(\mathscr{E}, \nabla: \mathscr{E}\to \mathscr{E}\otimes \Omega^1_{\mathcal{C}/S})$$ of $(E_0^{\text{an}}, \nabla_0^{\text{an}})$
to a relative flat vector bundle on $\pi^{-1}(U)$, where $U\subset S$ is any contractible analytic open set containing $s_0$,
such that (the isomorphism class of) the corresponding representation $\rho$ of the fundamental group
is constant. Explicitly, as $U$ is contractible, $\pi^{-1}(U)$ is naturally homotopy equivalent to $C_0$, so the composition $$\pi_1(\pi^{-1}(U))\overset{\sim}{\to} \pi_1(C_0)\overset{\rho_{\mathscr{E_0}, \nabla_0}}{\longrightarrow} GL(\mathscr{E}_{x(s_0)})$$ yields a local system on $\pi^{-1}(U)$, hence an (analytic) flat vector bundle. We call this the \emph{isomonodromic deformation} of $(\mathscr{E}_0, \nabla_0)$.
Such isomonodromic deformations are sometimes referred to as flat sections to the non-abelian Gauss-Manin connection.

Typically, if $S$ is not simply connected, the isomonodromic deformation does not extend to all of $\mathcal{C}/S$, even after \'etale base change. If it does, we say that $(\mathscr{E}, \nabla)$ admits an algebraic isomonodromic deformation.

\begin{dff}
Let $(\mathscr{E}, \nabla)$ be a flat vector bundle on a smooth proper curve $C$ of genus $g>1$. Let $\mathscr{C}_g\to \mathscr{M}_g$ be the universal curve over the Deligne-Mumford moduli stack of genus $g$ curves.  We say (following \cite{Cousin-Heu}) that $(\mathscr{E}, \nabla)$ admits a universal algebraic isomonodromic deformation if there exists an \'etale $U\to \mathscr{M}_g$ containing $[C]$ in its image such that $(\mathscr{E}, \nabla)$ admits an isomonodromic deformation to $U\times_{\mathscr{M}_g}\mathscr{C}_g$.
\end{dff}
By e.g.~\cite[Theorem A]{Cousin-Heu}, $(\mathscr{E}, \nabla)$ admits a universal algebraic isomonodromic deformation if and only if the orbit of $\rho_{\mathscr{E}, \nabla}$ under the mapping class group of $C$ is finite. (See Section 2.4 of \cite{Cousin-Heu} for an extension of these notions to the case of non-proper curves.)

Thus, using the Riemann existence theorem, Theorem \ref{main-thm} for surfaces without boundary admits a purely algebro-geometric statement:
\begin{thm}\label{main-thm-AG}
Let $C$ be a curve over $\mathbb{C}$ with $\chi(C)<0$, and let $(\mathscr{E}, \nabla)$ be a flat vector bundle on $C$. Suppose that for all finite \'etale maps of curves $f: C'\to C$, $f^*(\mathscr{E}, \nabla)$ admits a universal algebraic isomonodromic deformation. Then $(\mathscr{E}, \nabla)$ has finite monodromy.
\end{thm}
\subsection{The arithmetic setting and the $p$-curvature conjecture} The authors became interested in isomonodromic deformations
by way of the Grothendieck-Katz $p$-curvature conjecture \cite{Katz}. The strategy is based on an idea of Kisin, and is closely related to work of Papaioannou \cite{thanos}, Shankar \cite{Shankar}, and Patel-Shankar-Whang \cite{PSW}. 
Given a vector bundle with connection $(\mathscr{E}, \nabla)$ on a curve $C$ over an arbitrary field, we say that
$(\mathscr{E}, \nabla)$ \emph{admits a full set of algebraic sections}
if there exist some curve $C'$ and finite map $C' \rightarrow C$
such that the pullback of $(\mathscr{E}, \nabla)$ to $C'$ is spanned as an $\mathscr{O}_{C'}$-module by flat global sections.

Let $K$ be a finitely-generated field of characteristic zero, and take $C$ and $(\mathscr{E}, \nabla)$ as above.
We can spread this picture out to some integral domain $R\subset K$ with $\operatorname{Frac}(R)=K$,
and reduce modulo any maximal ideal $\mathfrak{m}$ of $R$. Let $(C_\mathfrak{m}, \mathscr{E}_\mathfrak{m}, \nabla_\mathfrak{m})$ denote the base change of 
the spreading-out of $(C, \mathscr{E}, \nabla)$ to $R/\mathfrak{m}$.
\begin{conjecture}[The $p$-curvature conjecture, Grothendieck-Katz]\label{p-curvature-1}
In order that $(\mathscr{E}, \nabla)$ admit a full set of algebraic sections,
it is necessary and sufficient that 
$(\mathscr{E}_\mathfrak{m}, \nabla_\mathfrak{m})$ admit a full set of algebraic sections for all $\mathfrak{m}$ in a non-empty open subset of $\operatorname{Spec}(R)$.
\end{conjecture}
Note that the hypothesis is independent of the chosen spreading-out, and that necessity above is clear. See \cite{Katz} for a discussion of Conjecture \ref{p-curvature-1}, and a proof in the case $(\mathscr{E}, \nabla)$ arises from the de Rham cohomology of a family of varieties over $C$, endowed with the Gauss-Manin connection.

The authors' main motivation for this paper is the observation that the hypothesis of the $p$-curvature conjecture (namely that $(\mathscr{E}_\mathfrak{m}, \nabla_\mathfrak{m})$ admit a full set of algebraic sections for all $\mathfrak{m}$ in a non-empty open subset of $\operatorname{Spec}(R)$) is stable under pullback. In particular, Theorem \ref{main-thm-AG} suggests the following reformulation of the $p$-curvature conjecture, in terms of the so-called non-abelian Gauss-Manin connection (i.e.~isomonodromic deformation). Let $C, K, \mathscr{E}, \nabla$ be as above. Choose an embedding $K\hookrightarrow \mathbb{C}$.
\begin{conjecture}\label{p-curvature-2}
If $(\mathscr{E}_\mathfrak{m}, \nabla_\mathfrak{m})$ admit a full set of algebraic sections for all $\mathfrak{m}$ in a non-empty open subset of $\operatorname{Spec}(R)$, then the flat vector bundle $(\mathscr{E}, \nabla)_{\mathbb{C}}$ on $C_{\mathbb{C}}$ admits a universal algebraic isomonodromic deformation.
\end{conjecture}
Conjectures \ref{p-curvature-1} and \ref{p-curvature-2} are equivalent by Theorem \ref{main-thm-AG}, and the well-known fact that the $p$-curvature conjecture may be reduced to the case of smooth proper curves of genus at least $2$.

\begin{remark}
Let $\mathscr{X}/\mathscr{O}_{K,S}$ be a smooth proper curve over the ring of $S$-integers of a number field $K$, and let $(\mathscr{E}, \nabla)$ be an arithmetic $\mathscr{D}_{\mathscr{X}/\mathscr{O}_{K,S}}$-module on $\mathscr{X}$ (this condition is a priori much stronger than the hypotheses of the $p$-curvature conjecture). Then one can use the main result of \cite{Shankar} to see that the analogue of Conjecture \ref{p-curvature-2} for such $(\mathscr{E}, \nabla)$ implies finiteness of monodromy.
\end{remark}

\begin{remark}
To connect our work to the $p$-curvature conjecture, one would like to know something about the behavior of $p$-curvature under isomonodromic deformation.
Unfortunately, it seems very difficult to say anything concrete here.
For instance, one might like to say that the condition of vanishing $p$-curvature is preserved under isomonodromic deformation.
But it's not clear how to make sense of this statement, since isomonodromic deformations don't in general exist integrally.

\end{remark}

\subsection{Plan of the proof of Theorem \ref{main-thm}}

    
    

%
    
   The argument is a proof by induction on the dimension of the representation.  Roughly speaking, if there is some $\gamma \in \pi_1(C)$ such that $\rho(\gamma)$ is not of finite order, we pass to a finite cover, make $\gamma$ a simple closed curve, and cut along $\gamma$.  After cutting, we show that the representation $\rho$ becomes reducible, so we can reduce the problem to a lower-dimensional case (on the cut surface).  If there is no such $\gamma$, we conclude by Lemma \ref{gln_finite_order}.
    

\subsection{Questions}
Our proof of Theorem \ref{main-thm} is geometric; Corollary \ref{main-cor} suggests that one might look for a purely group-theoretic proof. More generally, one might ask for an intrinsic characterization of those groups for which the analogue of Corollary \ref{main-cor} holds true.
\begin{dff}[Locally Extended Residually Finite (LERF)]
A group $G$ is said to be Locally Extended Residually Finite (LERF) if for every finitely-generated subgroup $H\subset G$, $H$ is closed in the profinite topology on $G$.
\end{dff}
\begin{question}
Suppose $G$ is finitely-generated and LERF. Let $$\rho: G\to GL_n(\mathbb{C})$$ be a representation such that for each finite-index subgroup $H\subset G$, the orbit of the isomorphism class of $\rho|_H$ under $\text{Out}(H)$ is finite. Does $\rho$ necessarily have finite image?
\end{question}
Note that Scott shows that surface groups are LERF \cite{Scott}; this fact is crucially used in the proof of Theorem \ref{main-thm}.

The next question was suggested to us by Junho Peter Whang; 
it asks whether, when the genus is large compared to the rank, we can eliminate the finite covers $\Sigma'$ from the statement of Theorem \ref{main-thm}.
A positive answer for $SL_2$ is given by \cite{BGMW}.
\begin{question}
Suppose $m$ is a positive integer.  Is the following statement true for all $\Sigma$ of sufficiently large genus:
For any representation $$\rho: \pi_1(\Sigma)\to GL_m(\mathbb{C}),$$
if the orbit of (the isomorphism class of) $f^*(\rho)$ under the mapping class group $MCG(\Sigma)$ is finite, then $\rho$ has finite image?
\end{question}
Finally, we propose two variants on our main theorem; we suspect both are true, and could be proven by similar methods, but we have not verified either.
\begin{question}
Does the statement of Theorem \ref{main-thm} remain true, if $\Sigma'$ is instead allowed to range over all branched covers of $\Sigma$ of degree 2?
\end{question}
\begin{question}
(Junho Peter Whang)
Suppose 
\[ \rho: \pi_1(\Sigma)\to GL_m(\mathbb{C}) \]
is an absolutely irreducible representation,
such that for each finite covering map $$f: \Sigma'\to \Sigma,$$ 
the orbit of (the isomorphism class of) $f^*(\rho)$ under the mapping class group $MCG(\Sigma')$ has compact closure in the character variety
classifying conjugacy classes of maps $\pi_1(\Sigma') \rightarrow GL_m(\mathbf{C})$.  Must $\rho$ be unitarizable?
(See \cite{Sikora} for the definition and properties of character variety.)
\end{question}

\subsection{Acknowledgements}
We would like to thank
Yves Andr\'e,
H\'el\`ene Esnault,
Benson Farb,
Mark Kisin,
Ananth Shankar,
Yunqing Tang, and
Junho Peter Whang
for helpful conversations. We would also like to thank the anonymous referees for their extremely useful comments. Litt is supported by NSF Grant DMS-2001196; Lawrence is supported by NSF Grant DMS-1705140.

\section{MCG-finiteness}
\subsection{Definitions}





Let $\Sigma$ be an orientable surface, possibly with boundary/punctures; let $p\in \Sigma$ be a point. Let $\text{MCG}(\Sigma)$ be the mapping class group of $\Sigma$.  (See \cite[Section 2.1]{FM}
for a discussion of mapping class groups.
In particular, recall that an element of $MCG(\Sigma)$ must fix $\partial \Sigma$ point-wise,
but may permute punctures.)

The natural map $$MCG(\Sigma)\to \text{Out}(\pi_1(\Sigma, p))$$ induces an action of $MCG(\Sigma)$ on the set of isomorphism classes of representations
of $\pi_1(\Sigma, p)$, as we now explain.

We say that two representations
\[ \rho_1, \rho_2 \colon \pi_1(\Sigma, p) \rightarrow GL_r(\mathbb{C}) \]
are \emph{isomorphic} 
if there is some $g \in GL_r(\mathbb{C})$
such that $\rho_2 = g \rho_1 g^{-1}$.

Any mapping class in $MCG(\Sigma)$ has a representative $T \colon \Sigma \rightarrow \Sigma$ that fixes the basepoint $p$.
Then $T$ gives an automorphism $\pi_1(T)$ of $\pi_1(\Sigma)$;
we denote by $T^* \rho$ the representation of $\pi_1(\Sigma)$
obtained by precomposing $\rho$ with $\pi_1(T)$.
If $T'$ is another representative of the same mapping class, also fixing $p$,
then $T'$ and $T$ agree up to an inner automorphism of $\pi_1(T)$,
so $(T')^* \rho$ is isomorphic (i.e.\ conjugate) to $T^* \rho$.

\begin{dff}
\label{mcg_finite}
Say a representation $$\rho: \pi_1(\Sigma, p)\to GL_r(\mathbb{C})$$ is \emph{MCG-finite} if the orbit of its isomorphism class under the action of $MCG(\Sigma)$ is finite. 

Say $\rho$ is \emph{universally MCG-finite} if, for any finite covering map $\Sigma' \rightarrow \Sigma$, its pullback to $\Sigma'$ is MCG-finite.
\end{dff}
\begin{remark}\label{birman-remark}
One way to produce MCG-finite representations is as follows. Let $\Sigma$ be a closed orientable surface, $p\in \Sigma$ a point, and consider the Birman exact sequence $$1\to \pi_1(\Sigma, p)\to MCG(\Sigma\setminus p)\to MCG(\Sigma)\to 1.$$ Then if $\rho: MCG(\Sigma\setminus p)\to GL_m(\mathbb{C})$ is any representation, $\rho|_{\pi_1(\Sigma, p)}$ is MCG-finite, by the normality of $\pi_1(\Sigma, p)$ in $MCG(\Sigma\setminus p)$. See Example \ref{parshin-example} for an example of a representation $\rho$ of $MCG(\Sigma\setminus p)$ such that $\rho|_{\pi_1(\Sigma, p)}$ has infinite image.
\end{remark}

Our goal in this paper is to show that \emph{universally} MCG-finite representations have finite image.

In order to prove this result, we will need a refined notion of universal MCG-finiteness for surfaces obtained by cutting a given surface along special collections of simple closed curves.

\begin{dff}
Let $\Sigma$ be an orientable surface, possibly with boundary/punctures. Let $\gamma_1, \cdots, \gamma_r$ be disjoint simple closed curves on $\Sigma$. We say that $\{\gamma_1, \cdots, \gamma_r\}$ \emph{are not jointly separating} if $\Sigma \setminus \bigcup_i \gamma_i$ is connected.
\end{dff}

\begin{figure}[h]
\centering
\includegraphics[scale=.8]{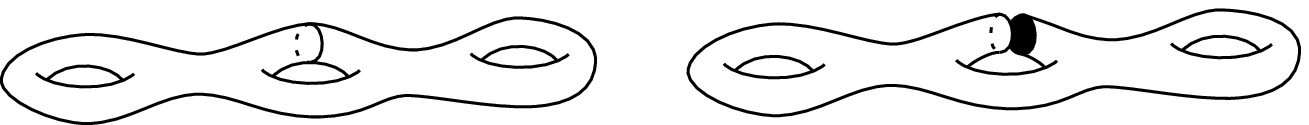}
 \caption{The surfaces $\Sigma$ and $\Ccut$.}
\end{figure}

Suppose $\Sigma$ is an orientable surface and $\{\gamma_1,\cdots, \gamma_r\}$ is a collection of disjoint curves in $\Sigma$. Then we define $\Ccut(\gamma_1, \cdots, \gamma_r)$ to be the surface with with boundary obtained by removing an $\epsilon$-neighborhood of $\bigcup_i \gamma_i$ from $\Sigma$.

\begin{dff}
Let $\Sigma$ be an orientable surface, possibly with boundary/punctures. Let $(\gamma_1, \cdots, \gamma_r)$ be a collection of simple closed curves in $\Sigma$ which are not jointly separating, with $r \geq 0$. Then a representation $$\rho: \pi_1(\Ccut(\gamma_1, \cdots, \gamma_r))\to GL_m(\mathbb{C})$$ is \emph{universally MCG-finite relative to $\Sigma$} if for each finite covering space $f: \Sigma'\to \Sigma$, and each connected component $X$ of $\Ccut'(f^{-1}(\gamma_1), \cdots, f^{-1}(\gamma_r))$, the representation $\rho|_{\pi_1(X)}$ is MCG-finite.
\end{dff}
We will in fact prove a version of Theorem \ref{main-thm} in the relative setting:
\begin{thm}\label{strong-theorem}
Let $\Sigma$ be an orientable surface, possibly with boundary/punctures. Let $(\gamma_1, \cdots, \gamma_r)$ be a collection of simple closed curves in $\Sigma$ which are not jointly separating. Suppose that $\chi(\Ccut(\gamma_1, \cdots, \gamma_r))<0$  and let $$\rho: \pi_1(\Ccut(\gamma_1, \cdots, \gamma_r))\to GL_m(\mathbb{C})$$ be a representation which is universally MCG-finite relative to $\Sigma$. Then $\rho^{\text{ss}}$ has finite image.
\end{thm}
Here $\rho^{\text{ss}}$ is the semi-simplification of $\rho$.

Theorem \ref{main-thm} will follow from Theorem \ref{strong-theorem} by Lemma \ref{exact_seq}; the modified statement of Theorem \ref{strong-theorem} is more amenable to geometric constructions, e.g.~to cutting the surface in question.

It is not clear to us whether semi-simplification is necessary in Theorem \ref{strong-theorem}.
\begin{question}\label{ss-question}
Let $\Sigma$ be an orientable surface, possibly with boundary/punctures. Let $(\gamma_1, \cdots, \gamma_r)$ be a collection of simple closed curves in $\Sigma$ which are not jointly separating. Suppose that $\chi(\Ccut(\gamma_1, \cdots, \gamma_r))<0$  and let $$\rho: \pi_1(\Ccut(\gamma_1, \cdots, \gamma_r))\to GL_m(\mathbb{C})$$ be a representation which is universally MCG-finite relative to $\Sigma$. Does $\rho$ necessarily have finite image?
\end{question}
\subsection{Basic properties of MCG-finiteness}
\begin{lem}\label{subquot} The following hold:
\begin{itemize}
\item The semi-simplification of an MCG-finite representation is MCG-finite.
\item Any sub-quotient of the semi-simplification of an MCG-finite representation is MCG-finite.

\item Let $\Sigma$ be an orientable surface, possibly with boundary/punctures.  Let $(\gamma_1, \cdots, \gamma_r)$ be a collection of simple closed curves in $\Sigma$ which are not jointly separating. Then any sub-quotient of the semi-simplification of a representation of $\pi_1(\Ccut(\gamma_1, \cdots, \gamma_r))$ which is universally MCG-finite relative to $\Sigma$ is universally MCG-finite relative to $\Sigma$.
\end{itemize}
\end{lem}

\begin{proof}
The first statement follows from the fact that for representations $V,W$, we have $V^{ss}\simeq W^{ss}$ if $V\simeq W$; 
thus, if the mapping class group orbit of $V$ is finite, then the same is true of $V^{ss}$.

To see the second statement, take $V$ an MCG-finite representation and $W$ a sub-quotient of $V^{ss}$.  Then for any $T \in MCG(\Sigma)$, we know that $T^* W$ is a sub-quotient of $T^* V^{ss}$.  The set of sub-quotients of $\{T^*V^{ss}\}_{m\in \mathbb{Z}}$ is finite by MCG-finiteness; hence there are only finitely many possibilities for $T^*W$, as desired. 

The third claim is proved in exactly the same way.
\end{proof}

Next we will see that
the property of universal MCG-finiteness is preserved
when we cut $\Sigma$ along a non-separating simple curve $\gamma$.

\begin{lemma}
\label{umf_cut}
Let $\Sigma$ be an orientable surface, possibly with boundary, and let $(\gamma_1, \cdots, \gamma_r, \cdots, \gamma_s)$ be a collection of disjoint simple closed curves on $\Sigma$ which are not jointly separating.
Suppose $\rho: \pi_1(\Ccut(\gamma_1, \cdots, \gamma_r)) \rightarrow \GL(V)$ is universally MCG-finite relative to $\Sigma$ (or, if $r = 0$, that $\rho$ is universally MCG-finite).
Then the restriction $\rho | _{\Ccut(\gamma_1, \cdots, \gamma_s)}$ is universally MCG-finite relative to $\Sigma$.
\end{lemma}

\begin{proof}
This is immediate from the definitions. Indeed, let $f: \Sigma'\to \Sigma$ be any finite covering map, and let $X$ be a component of $\Ccut'(f^{-1}(\gamma_1), \cdots, f^{-1}(\gamma_s))$. Then any mapping class on $X$ extends to a mapping class on $\Sigma'$ (recall that mapping classes on surface with boundary must fix the boundary, by definition). 
The result follows.
\end{proof}

Next we'll work out a concrete consequence of MCG-finiteness.
Suppose $\rho$ is MCG-finite, and
let $\gamma$ be a simple closed curve on $\Sigma$,
not passing through the base-point $p$.
Consider the Dehn twist $T_{\gamma}$.
MCG-finiteness implies that for some $m>0$, there is an isomorphism
\begin{equation}\label{isom_exists} (T_{\gamma}^m)^* \rho \cong \rho.
\end{equation}

The explicit geometric construction of the Dehn twist $T_{\gamma}$ gives a bona fide automorphism (not just an outer automorphism)
of $\pi_1(\Sigma, p)$.
Hence, there is a preferred choice of isomorphism 
between the underlying vector spaces of
$(T_{\gamma}^m)^*\rho$ and $\rho$;
we'll use this isomorphism without comment
in what follows. 

The data of the isomorphism (\ref{isom_exists})
is the same as that of a linear map $g: V \overset{\sim}{\rightarrow} V$, intertwining the actions of $\pi_1(\Sigma, p)$ via $(T^m_\gamma)^*\rho$ and $\rho$ --- we call $g$ an ``intertwining operator." Thus we have shown:

\begin{lemma}
\label{dehn_conjugate}
Suppose $\rho$ is MCG-finite and $\gamma$ is a simple closed curve on $C$.  Then there exist a positive integer $m$ and an automorphism $g$ of $V$ such that, for every $\delta \in \pi_1(\Sigma, p)$, we have
\[ \rho (T_\gamma^m \delta) = g \rho(\delta) g^{-1}. \]
\end{lemma}
Note that if $\rho$ is simple, the intertwining operator $g$ above is unique up to scaling, by Schur's lemma.

\begin{dff}
\label{mod_bdry}
Let $(\Sigma,p)$ be an orientable surface, possibly with boundary/punctures. 
We denote by $H_1(\Sigma) = H_1(\Sigma, \mathbb{Z})$ the homology of $\Sigma$ with \emph{integral} coefficients.

Let $\overline{H_1(\Sigma)}$ denote the quotient of $H_1(\Sigma)$ by the span of classes of boundary components and loops around punctures.

We say that $\gamma\in \pi_1(\Sigma, p)$ is \emph{nontrivial modulo boundary} (in $H_1(\Sigma)$)
if the class of $\gamma$ in $\overline{H_1(\Sigma)}$
is non-trivial. 

\end{dff}

\begin{remark}
Note that $\overline{H_1(\Sigma)}$ is naturally identified with the first homology of the compact orientable surface without boundary obtained by filling in the punctures of $\Sigma$ and gluing caps onto the boundary components of $\Sigma$. In particular, $\overline{H_1(\Sigma)}$ has a well-defined intersection product $\langle -, -\rangle$; we abuse notation and denote the induced product on $H_1(\Sigma)$ (which is in general no longer non-degenerate) via $\langle - , - \rangle$ as well. 
\end{remark}

We now show that the property of having finite image is preserved under extensions of universally MCG-finite representations.

\begin{lemma}
\label{exact_seq}
Suppose $\chi(\Sigma)<0$, and $\rho: \pi_1(\Sigma) \rightarrow \GL(V)$ fits in an exact sequence of representations
\[0 \rightarrow \rho_1 \rightarrow \rho \rightarrow \rho_2 \rightarrow 0,\]
where \(\rho_1\) and \(\rho_2\) have finite image,
and suppose $\rho$ itself is universally MCG-finite.
Then $\rho$ has finite image.  
\end{lemma}

\begin{proof}
Let $V_1$ be the vector space underlying $\rho_1$ and $V_2=V/V_1$ the vector space underlying $\rho_2$. By passing to a finite cover of $\Sigma$, we may assume $\rho_1$ and $\rho_2$ are in fact trivial.

The representation \( \rho : \pi_1(\Sigma, p) \rightarrow GL(V)\)
factors through the group 
\[ \Aut_{V_1, V_2} (V) \]
of automorphisms \(g\) of \(V\) fixing \(V_1\)
and acting trivially on both  \(V_1\)  and \(V_2\).

Given $f\in \text{Hom}(V_2, V_1)$, let $\bar f: V\to V$ be the composition $$V\twoheadrightarrow V_2\overset{f}{\to} V_1\hookrightarrow V.$$
It is a standard fact from linear algebra that the map $f\mapsto \bar f+\text{Id}$ is an isomorphism
\[ \Hom(V_2, V_1) \cong \Aut_{\mathrm{V_1, V_2}} (V).\]

Hence the representation \(\rho\) has abelian image, so it factors through a homomorphism
\[\sigma_{\Sigma}: H_1(\Sigma) \rightarrow \Hom(V_2, V_1).\]
The same remains true if we pull back to any finite cover \(\Sigma' \rightarrow \Sigma\),
and the resulting diagram
\[
\xymatrix{
\sigma_{\Sigma'}: H_1(\Sigma') \ar[r] \ar[d] & \Hom(V_2, V_1) \ar[d]^{=} \\
\sigma_{\Sigma}: H_1(\Sigma) \ar[r] & \Hom(V_2, V_1) \\
}
\]
commutes.

It suffices to show that $\sigma_{\Sigma}$ is identically zero. Assume for a contradiction that \(\sigma_{\Sigma} \neq 0\).  Then it is possible
to produce a cover \(\Sigma'\) and two classes \(\gamma_1, \gamma_2 \in H_1(\Sigma')\)
such that
\begin{itemize}
\item $\sigma_{\Sigma'}(\gamma_1) \neq 0$
\item $\sigma_{\Sigma'}(\gamma_2) = 0$
\item $\gamma_1$ and $\gamma_2$ have intersection number \(\langle \gamma_1, \gamma_2\rangle = i \neq 0\). 
\item $\gamma_1$ is represented by a simple closed loop.
\end{itemize}

To see this, we first choose some nontrivial cover $\Sigma'$ of $\Sigma$ so that $\text{genus}(\Sigma')>\text{genus}(\Sigma)$ (using that $\chi(\Sigma)<0$),
and let $\gamma_2$ be a class in $H_1(\Sigma')$, 
nontrivial modulo boundary, 
and in the kernel of the $H_1(\Sigma') \rightarrow H_1(\Sigma)$.

Then, we choose a simple closed loop $\gamma_1$ in $\Sigma'$ whose class in $H_1(\Sigma')$ satisfies: 
$\sigma_{\Sigma'}(\gamma_1) \neq 0$ and $\langle \gamma_1, \gamma_2 \rangle \neq 0$. This is possible as the set of $\gamma$ failing one of these conditions is a union of two proper linear subspaces of $H_1(\Sigma')$, and thus does not contain every primitive element in $H_1(\Sigma')$; but any primitive element is represented by a simple closed loop by e.g.\ \cite[Proposition 6.2]{FM}.

Now apply Lemma \ref{dehn_conjugate} to the Dehn twist \(T_{\gamma_1}\).
We find some \(m\) such that
\[(T_{\gamma_1}^m)^* \rho \cong \rho\] and hence $$(T_{\gamma_1}^m)^*\sigma_{\Sigma'}=\sigma_{\Sigma'}.$$
On the other hand, we have
\[ \sigma_{\Sigma'}(\gamma_2) = 0 \]
and
\[ (T_{\gamma_1}^m)^* \sigma_{\Sigma'} (\gamma_2) = \sigma_{\Sigma'}(\gamma_2 \gamma_1^{im}) = \sigma_{\Sigma'}(\gamma_1^{im}) \neq 0,  \]
so the two representations $(T_{\gamma_1}^m)^*\rho$ and $\rho$ cannot be isomorphic. This is our desired contradiction.
\end{proof}
\begin{remark}
We do not know if an analogue of Lemma \ref{exact_seq} holds in the relative setting.
\end{remark}


\section{Proof of Theorems \ref{main-thm} and \ref{strong-theorem}}
Before proceeding to the proof of Theorem \ref{main-thm}, we record a few useful lemmas.

\subsection{Some useful lemmas} First, we prove some variants 
on a lemma that appears in \cite{Tits} (Proposition 2.5).

Recall that a matrix is called \emph{quasi-unipotent} if some power of it is unipotent.
\begin{lem}
\label{quasi-unipotent1}
Let $\rho: G \rightarrow \GL(V)$
be a representation of a finitely-generated group on a complex vector space.
Suppose that for every $g \in G$, the transformation $\rho(s)$ is quasi-unipotent.
Then the semisimplification $\rho^{ss}$ has finite image.
\end{lem}

\begin{proof}
First, suppose $\rho$ is simple, so $\rho(G)$ spans the algebra $\End(V)$ as a $\mathbb{C}$-vector space, by Burnside's Theorem \cite[Theorem 27.4]{CR}. 
Let $g_1, \ldots, g_{n^2}$ be elements in $G$ such that $\{\rho(g_i)\}$ forming a basis for $\End(V)$,
and let $e_1, \ldots, e_{n^2}$ be the dual basis under the trace pairing.
For $g \in G$, we have
\[ \rho(g) = \sum_i \Tr(\rho(g_i)^{-1} \rho(g)) e_i. \]

By the proof of \cite[Lemma 2.4]{Tits}, 
$\Tr(g)$ takes on only finitely many values 
as $g$ ranges over all elements of $G$. 
%
We conclude that $\rho(G)$ is finite.

The general result follows immediately from the simple case.
\end{proof}

%

\begin{lem}
\label{quasi-unipotent_subgroup}
Let $G$ be a finitely-generated group, and $H \subseteq G$ a subgroup of finite index. 
Suppose $\rho: G \rightarrow \GL(V)$ is a representation
such that $(\rho |_H)^{ss}$ has finite image.
Then $\rho^{ss}$ also has finite image.
\end{lem}

\begin{proof}

We know that $\rho(g)$ is quasi-unipotent for $g \in H$.
Now let $g \in G$ be arbitrary; there exists some $n > 0$ such that $g^n \in H$.  Since $\rho(g)^n$ is quasi-unipotent, $\rho(g)$ itself is as well.
We conclude by Lemma \ref{quasi-unipotent1} that $\rho^{ss}$ has finite image.
\end{proof}

\begin{lem}
\label{gln_finite_order}
Let $(\Sigma, p)$ be a pointed orientable surface (possibly with punctures or boundary),
and let $\rho: \pi_1(\Sigma, p) \rightarrow \GL(V)$ be a simple representation of $\pi_1(\Sigma, p)$
on a complex vector space.  (Here simple means that
$V$ has no proper nontrivial $\rho$-stable subspace $W$.)

Assume that $\Sigma$ has genus at least 1.  
Suppose that, for every $\gamma \in \pi_1(\Sigma, p)$ nontrivial modulo boundary in $H_1(\Sigma)$ (Def.\ \ref{mod_bdry}),
$\rho(\gamma)$ is quasi-unipotent.
Then $\rho$ has finite image.
\end{lem}

\begin{proof}
Let $G=\pi_1(\Sigma, p)$.

Since $\rho$ is simple, $\rho(G)$ spans the algebra $\End(V)$ as a $\mathbb{C}$-vector space. 
Let $g_1, \ldots, g_{n^2}$ be elements in $G$ such that $\{\rho(g_i)\}$ forming a basis for $\End(V)$,
and let $e_1, \ldots, e_{n^2}$ be the dual basis under the trace pairing.
For $g \in G$, we have
\[ \rho(g) = \sum_i \Tr(\rho(g_i)^{-1} \rho(g)) e_i. \]

Let $\Gamma\subset G$ be the subset consisting of those elements $\gamma\in G$ such that $g_i^{-1}\gamma$ is non-trivial mod boundary for all $i$.

By the proof of \cite[Lemma 2.4]{Tits}, 
$\Tr(\gamma)$ takes on only finitely many values 
as $\gamma$ ranges over all elements of $\pi_1(\Sigma, p)$
nontrivial modulo boundary in $H_1$. Thus there is a finite subset $G_0 \subseteq \End V$
such that, if $\gamma$ is contained in $\Gamma$,
then in fact $\rho(\gamma) \in G_0$.

Now we claim that any $\gamma$ can be written as a product $\gamma=\gamma_1\gamma_2$, with in $\gamma_1, \gamma_2\in \Gamma$.  Let $\bar g_j$ be the image of $g_j$ in $H_1(C)$. It is enough to show that any $h\in H_1(\Sigma)$ can be written as $h=h_1+h_2$, where $h_j-\bar{g_i}$ has non-zero image in $\overline{H_1(\Sigma)}$ for all $i,j$.  But now we may choose $h_1$ to be any element of $H_1(\Sigma)$ whose image in $\overline{H_1(\Sigma)}$ is not equal to that of $\bar g_i$ or $h-\bar g_i$ for any $i$, and set $h_2=h-h_1$.

We conclude that $\rho(G) \subseteq G_0 G_0$, so $\rho(G)$ is finite.
\end{proof}

\begin{cor}
\label{gln_quasi_finite}
Let $(\Sigma, p)$ be a pointed orientable surface (possibly with punctures or boundary),
and let $\rho: \pi_1(\Sigma, p) \rightarrow \GL(V)$ be any representation of $\pi_1(\Sigma, p)$
on a complex vector space.  

Assume that $\Sigma$ has genus at least 1.  
Suppose that, for every $\gamma \in \pi_1(\Sigma, p)$ nontrivial modulo boundary in $H_1(\Sigma)$,
$\rho(\gamma)$ is quasi-unipotent.
Then the semi-simplification $\rho^{ss}$ of $\rho$ has finite image.
\end{cor}
\begin{proof}
Immediate from Lemma \ref{gln_finite_order}.
\end{proof}


%

Not every class $\gamma \in \pi_1(\Sigma, p)$ can be represented by a simple curve
(i.e.\ a curve with no self-intersection).
But any $\gamma$ becomes a simple curve after pullback to a cover.

The result we need is a simple application of a theorem of Scott \cite[Theorem 3.3]{Scott}. Scott's theorem allows one to find covers such that a given curve lifts to a simple closed curve; we observe that one may do so while keeping it disjoint from a collection of other curves.  (A detailed exposition of Scott's proof may be found in \cite{Patel}.)

\begin{lem}
\label{make_curve_simple}
Let $(\Sigma, p)$ be a pointed orientable surface with $\chi(\Sigma) < 0$, and $\gamma_1, \ldots, \gamma_r$ simple closed curves on $\Sigma$, not passing through $p$ and not jointly separating.
Suppose $\gamma \in \pi_1(\Sigma, p)$ is represented by a closed curve that is disjoint from  $\gamma_1, \ldots, \gamma_r$.

Then there exists a finite cover $f \colon (\Sigma', p') \rightarrow (\Sigma, p)$ such that the subgroup $\pi_1(\Sigma',  p') \subseteq \pi_1 (\Sigma, p)$ contains $\gamma$,
and the class $\gamma \in \pi_1(\Sigma',  p')$ is represented by a \emph{simple} closed curve, disjoint from the curves $f^{-1} (\gamma_i)$.
\end{lem}

\begin{proof}
Scott's theorem (\cite[Theorem 3.3]{Scott}) shows that
there is some $f \colon  (\Sigma', p') \rightarrow (\Sigma, p)$ such that
$\gamma \in \pi_1(\Sigma', p')$ is represented by a simple closed curve in $\Sigma'$.
We need to show that $\gamma$ may be taken to avoid the curves $f^{-1} (\gamma_i)$.

We can put a hyperbolic metric on $\Sigma'$ in such a way that each $\gamma_i$ is a geodesic.
The universal cover $\tilde{\Sigma}$ of $\Sigma'$
is a convex region in the hyperbolic plane $\mathbf{H}$,
bounded by lines.
(If $\Sigma$ -- hence $\Sigma'$ -- has no boundary components,
then the universal cover is all of $\mathbf{H}$.  See \cite[end of \S 5]{Patel}.)

The curves $\gamma_i$ lift to lines in $\tilde{\Sigma} \subseteq \mathbf{H}$;
let $L$ be the union of these lines,
and let $X$ be the connected component of $p$ in $\tilde{\Sigma} - L$.
Again, $X$ is a convex region in $\mathbf{H}$, bounded by a union of lines.

The fundamental group of $\Sigma'$ acts on
the universal cover $\tilde{\Sigma}$ by deck transformations,
which extend to translations of $\mathbf{H}$.
Let $g$ be the translation of $\mathbf{H}$ corresponding to $\gamma$.
Since $\gamma$ is represented by a curve in $\Sigma$
that avoids the curves $\gamma_i$, we find that
$g^n p' \in X$ for all $n \in \mathbf{Z}$.
The translation $g$ acts on $\partial \mathbf{H}$
with two fixed points $g^{\pm \infty} p'$.
By a limiting argument, both these fixed points
lie in $\partial X$, so the geodesic $\ell$ between them lies in $X$.

Now $\ell$ descends to a curve isotopic to $\gamma$ in $\Sigma'$; this is again a simple closed curve by \cite[Proposition 1.6]{FM}, 
and it avoids the curves $f^{-1} (\gamma_i)$ by construction. 
\end{proof}

Finally, we observe:
\begin{lem}
\label{nonsimple}
Let $V$ be a complex vector space, and
let $g \in \GL(V)$ be an element not equal to a scalar matrix.
Then there exists a nontrivial proper subspace $W \subset V$
such that, if $h \in \GL(V)$ commutes with $g$,
then $h$ stabilizes $W$.
\end{lem}

\begin{proof}
Take $W$ to be any nontrivial eigenspace for any eigenvalue of $g$.
\end{proof}

\subsection{Proof of the main theorem}\label{thm-pf} We may now proceed with the proof of our main result. First we prove the result for relatively universally MCG-finite representations.

\begin{proof}[Proof of Theorem \ref{strong-theorem}]
For the reader's convenience, we briefly recall the theorem we are trying to prove. Namely, let $\Sigma$ be an orientable surface, possibly with boundary/punctures, and let $\gamma_1, \ldots, \gamma_r$ be disjoint simple closed curves on $\Sigma$, not jointly separating.  Let $p\in \Ccut(\gamma_1, \cdots, \gamma_r)$ be a point. Suppose $\pi_1(\Ccut(\gamma_1, \cdots, \gamma_r), p)$ is nonabelian, and we have a representation 
\[\rho: \pi_1(\Ccut(\gamma_1, \cdots, \gamma_r), p)\to GL(V)\]
that is universally MCG-finite relative to $\Sigma$.
Then we wish to show that the semi-simplification $\rho^{\text{ss}}$ of $\rho$ has finite image.

Write $\Ccut = \Ccut(\gamma_1, \cdots, \gamma_r)$. We will proceed by induction on the rank of $\rho$.

By Lemma \ref{quasi-unipotent_subgroup}, there is no harm in passing to a finite cover of $\Sigma$. Hence we may assume that $\Ccut$ is of genus at least 2.

Let $S \subseteq \pi_1(\Ccut, p) $ be the set of elements that are nontrivial modulo boundary in $H^1(\Ccut)$.  We will show that for $\gamma\in S$, $\rho(\gamma)$ is quasi-unipotent.  

Fix $\gamma\in S$. Pulling back to a further finite cover $\Sigma'$ of $\Sigma$, we may assume by Lemma \ref{make_curve_simple}
that $\gamma$ is isotopic to a simple closed curve $\gamma_{r+1}$, such that $\gamma_1, \ldots, \gamma_{r+1}$ are not jointly separating in $\Sigma$ (as $\gamma$ is non-trivial mod boundary in $\Ccut$). 

By Lemma \ref{dehn_conjugate}, there exist a positive integer $m$ and an automorphism $g$ of $V$ such that, for every $\delta \in \pi_1(\Ccut, p)$, we have
\[ \rho (T_{\gamma_{r+1}}^m \delta) = g \rho(\delta) g^{-1}. \]

First, assume $g$ is a scalar matrix.  Then we have
\[ \rho(T_{\gamma_{r+1}}^m \delta) = \rho(\delta) \]
for all $\delta$.  Taking $\delta$ a curve in $\Ccut$ which meets $\gamma_{r+1}$ exactly once, and is transverse to $\gamma_{r+1}$ at that point
(possible because $\gamma_{r+1}$ is non-separating),
we have $$\rho(\delta)\rho(\gamma_{r+1})^m=\rho(\delta\gamma_{r+1}^m)=\rho(T^m_{\gamma_{r+1}}\delta)=\rho(\delta)$$
and we find that
\[ \rho(\gamma_{r+1})^m = 1, \]
as desired. In particular, if $\dim(V)=1$, $g$ is always a scalar matrix, and so we may conclude the base case of the induction.

Now, assume $g$ is not a scalar matrix.
Let $\Ccut'$ be the surface obtained from $\Ccut$ by cutting along $\gamma_{r+1}$, 
with base-point $p$ not on one of the boundary components coming from $\gamma_{r+1}$.
Since $\gamma_{r+1}$ is a non-separating closed curve, and $\Ccut$ has genus at least 2, $\Ccut'$ is again a surface with nonabelian $\pi_1$,
and our local system on $\Ccut$ pulls back to a local system on $\Ccut'$.

The representation
\[ \rho|_{\pi_1(\Ccut', p)}: \pi_1(\Ccut', p) \rightarrow \GL(V) \]
factors through the centralizer of $g$, as  $\pi_1(\Ccut', p)$ is generated by a loop isotopic to $\gamma_{r+1}$ and by loops in $\Sigma$ not intersecting $\gamma_{r+1}$.
By Lemma \ref{nonsimple}, there is thus a non-trivial proper subspace $W \subset V$
which is stable under this representation. 

By Lemma \ref{umf_cut}, $\rho|_{\pi_1(\Ccut', p)}$ is universally MCG-finite relative to $\Sigma'$. Hence by Lemma \ref{subquot}, $W^{\text{ss}}$ and $(V/W)^{\text{ss}}$ are universally MCG-finite relative to $\Sigma'$ and thus by the inductive hypothesis, the representations on $W^{\text{ss}}$ and $(V/W)^{\text{ss}}$ have finite image. Thus $\rho(\gamma_{r+1})$ is quasi-unipotent.

Now we conclude the result by Corollary \ref{gln_quasi_finite}.
\end{proof}
We now deduce the main result of the paper.
\begin{proof}[Proof of Theorem \ref{main-thm}]
By Theorem \ref{strong-theorem}, $\rho^{ss}$ has finite image. Now we may conclude by Lemma \ref{exact_seq}, by induction on the number of components in the composition series for $\rho$.
\end{proof}
\subsection{An example: the Parshin representation}
In the course of proving Theorem \ref{main-thm},
we also prove the following intermediate result.

Suppose $$\rho:\pi_1(\Sigma,p)\to GL(V)$$ is MCG-finite, and $\gamma$ is some simple closed non-separating curve in $\Sigma$. Let $\Ccut$ be the surface obtained by cutting $\Sigma$ along $\gamma$.
Take $m$ such that
\[(T_\gamma^m)^* \rho \cong \rho,\]
and let
\[ g \colon V \rightarrow  V \]
be an intertwining operator such that
\[ \rho(T_{\gamma}^m \delta) = g \rho(\delta) g^{-1}. \]

Then for any $\delta$ disjoint from $\gamma$, $\rho(\delta)$ must commute with $g$.
In particular, if $g$ is not a scalar matrix, then the restriction of $\rho$ to $\Ccut$ is reducible.

As an example, we'll see what this looks like for $\rho$ the Parshin representation. The idea goes back to Kodaira and Parshin \cite[Appendix]{Parshin}; 
an explicit construction with the properties we describe is given in \cite[Section 7]{LV}. 
\begin{ex}\label{parshin-example}
Let $\Sigma = C$ be a complex algebraic curve of genus at least $2$,
and let $p$ be a point of $C$.
There are finitely many isomorphism classes of degree-3 covers $Y_p^i$ of $C$, 
branched at $p$ and nowhere else,
and having Galois group $S_3$.
There is an algebraic family $\pi \colon Y \rightarrow C$ whose fiber over any point $p \in C(\mathbf{C})$
is the disjoint union of the curves $Y_p^i$. Let $Y\to C'\to C$ be the Stein factorization of $\pi$.
The cohomology of this family $R^1 \pi_* (\mathbf{C}_Y)$
gives a local system on $C$;
the corresponding representation $\rho$ is exactly the monodromy representation of $\pi_1(C)$ on the cohomology of a fiber.

First, note that $\rho$ is MCG-finite, as by construction it extends to a representation of $MCG(C\setminus\{p\})$ (see Remark \ref{birman-remark}).

Next, note that $\rho$ is virtually reducible: on a finite-index subgroup of $\pi_1(C)$ (namely, $\pi_1(C')$), it splits up as a direct sum of the $H^1(Y_p^i)$. Each $H^1(Y_p^i)$ is virtually irreducible; this follows from the big monodromy result \cite[Lemma 4.3]{LV}. In particular, $\rho$ has infinite image; it is thus not universally MCG-finite.

Now fix a simple closed curve $\gamma$ on $C$, not passing through $p$, and take $m$ sufficiently divisible.  
Then $\gamma^m$ lifts to a disjoint  union of simple closed curves on each cover $Y_p^i$.
Thus $T_{\gamma}^m$
acts unipotently, as a product of commuting Dehn twists.

The corresponding $g$ is also unipotent: it's given by the action of a Dehn twist about $\gamma$, as an element of MCG(C).

Let $V_i$ be the subspace of $H^1$ 
dual to the subspace of $H_1$
spanned by the components of the lift of $\gamma^m$ to $Y^i_p$.
Then $\sum_i V_i$ (the direct sum of the $V_i$) is a $g$-stable subspace of $\rho | _{\pi_1(\Ccut)}$.

In particular, we see that $\rho|_{\pi_1(\Ccut)}$ is reducible, as expected.
\end{ex}

\subsection{An example from topological quantum field theory}
Unlike universally MCG-finite representations, MCG-finite representations can be quite interesting. 
\begin{ex}
In \cite[Thm.\ 1.1]{KS}, Koberda and Santharoubane construct representations $\rho$ of $\pi_1(\Sigma)$ with the following properties:
\begin{itemize}
    \item $\rho$ has infinite image.
    \item $\rho(\gamma)$ has finite order, for any simple closed curve $\gamma$.
    \item $\rho$ is MCG-fixed (\cite[\S 1.2]{KS}).
\end{itemize}
\end{ex}
This example is particularly relevant to the strategy outlined in the introduction for proving the $p$-curvature conjecture; it shows that it does not suffice to study the monodromy of simple closed loops in a surface.
\subsection{An example: a unipotent, MCG-finite representation}
We conclude with an example philosophically relevant to Question \ref{ss-question}. Namely, for every surface $\Sigma$ with $\chi(\Sigma)<0$, we observe that there exist non-trivial unipotent representations of the fundamental group of $\Sigma$ which are stable under the action of the mapping class group of $\Sigma$.
\begin{ex}
Let $\Sigma$ be a surface with $\chi(\Sigma)<0$, and let $p\in \Sigma$ be a point, so that $\pi_1(\Sigma,p)$ is non-abelian. Let $\mathbb{Q}[\pi_1(\Sigma,p)]$ be the group algebra of $\Sigma$ and $\mathscr{I}\subset \mathbb{Q}[\pi_1(\Sigma,p)]$ the augmentation ideal. Then we claim that for any $n>1$, the representation of $\pi_1(\Sigma,p)$ on $V_n:=\mathbb{Q}[\pi_1(\Sigma,p)]/\mathscr{I}^n$ induced by the action of $\pi_1(\Sigma)$ on itself by conjugation is non-trivial, unipotent, and fixed by the action of the mapping class group.

Indeed, direct computation shows that these representations are non-trivial; they are unipotent as for each $i$, the action of $\pi_1(\Sigma)$ on $\mathscr{I}^i/\mathscr{I}^{i+1}$ is trivial. Finally, these representations are MCG-finite (indeed, they are fixed by $MCG(\Sigma)$) by Remark \ref{birman-remark}, as each $V_n$ is naturally a representation of $MCG(\Sigma\setminus p)$.
\end{ex}
\bibliographystyle{alpha}
\bibliography{mcg-finite}

\end{document}